\documentclass[11pt,letterpaper,reqno]{amsart}
\usepackage{amsfonts}
\usepackage{amsmath}
\usepackage{amsthm, amscd}
\usepackage{mathtools}
\usepackage{amssymb}
\usepackage{mathrsfs}
\usepackage{graphicx}
\usepackage{caption}
\usepackage{subcaption}
\usepackage{float}
\usepackage{xypic}
\usepackage[abs]{overpic}
\usepackage[alphabetic]{amsrefs}
\usepackage{etex}
\usepackage{tikz-cd}
\usepackage{hyperref}
\hypersetup{
	colorlinks,
	linkcolor=[rgb]{0,0,0.7},
	urlcolor=[rgb]{0,0,0.4},
	citecolor=[rgb]{0.4,0.1,0}
}

\allowdisplaybreaks

\theoremstyle{plain}\newtheorem{Theorem}{Theorem}[section]
\theoremstyle{plain}\newtheorem{Corollary}[Theorem]{Corollary}
\theoremstyle{plain}\newtheorem{Lemma}[Theorem]{Lemma}
\theoremstyle{plain}
\theoremstyle{plain}\newtheorem{Proposition}[Theorem]{Proposition}
\theoremstyle{plain}
\theoremstyle{plain}
\theoremstyle{plain}\newtheorem*{Claim*}{Claim}
\theoremstyle{plain}\newtheorem*{Theorem*}{Theorem}
\theoremstyle{plain}\newtheorem*{Lemma*}{Lemma}
\theoremstyle{plain}

\theoremstyle{remark}\newtheorem{remark}[Theorem]{Remark}
\theoremstyle{remark}
\theoremstyle{remark}\newtheorem*{Notation*}{Notation}

\numberwithin{equation}{section}

\DeclareMathOperator{\dist}{dist}

\DeclareMathOperator{\Ima}{Im}

\DeclareMathOperator{\Homeo}{Homeo}
\DeclareMathOperator{\Emb}{Emb}
\DeclareMathOperator{\Imm}{Imm}
\DeclareMathOperator{\Map}{Map}

\newcommand{\bZ}{\mathbb{Z}}

\author{Jianfeng Lin}
\address{Yau Mathematical Sciences Center, Tsinghua University, Beijing 100084, China}
\email{linjian5477@mail.tsinghua.edu.cn}
\author{Yi Xie}
\address{Beijing International Center for Mathematical Research, Peking University, Beijing 100871, China}
\email{yixie@pku.edu.cn}
\author{Boyu Zhang}
\address{Department of Mathematics, The University of Maryland at College Park, Maryland 20742, USA}
\email{bzh@umd.edu}

\title{The Dax isomorphism for topological $4$-manifolds}

\begin{document}

\maketitle 

\begin{abstract}
Let $X$ be an oriented topological $4$-manifold with boundary. We establish a Dax isomorphism for the fundamental group of the space of embedded arcs in $X$. Here, the embedding space can be either the simplicial set of locally flat embeddings or the space of topological embeddings endowed with the compact-open topology. When $X$ is smooth, we show that the space of smooth arcs and topological arcs have canonically isomorphic fundamental groups, and the topological Dax isomorphism agrees with the smooth Dax isomorphism. As a consequence, the homomorphisms that arise in the smooth Dax isomorphism do not depend on the smooth structure. 
\end{abstract}

\section{Introduction}

Dax \cite{Dax1972} studied the relations between the homotopy groups of embedding spaces and immersion spaces of smooth manifolds. For embeddings of arcs in $4$-manifolds, Gabai \cite{Gabai2021} gave a concrete formulation of the Dax isomorphism, which is an exact sequence that describes the fundamental group of the space of smoothly embedded arcs in a given $4$-manifold. Further properties of the Dax isomorphism were developed by Schwartz \cite{Schwartz2021}, Kosanovi\'c \cite{Kosanovic2024Dax}, and Kosanovi\'c--Teichner \cite{KosanovicTeichner2024Space}. The Dax isomorphism and its associated isotopy invariants have since become useful tools in the study of isotopy problems in dimension $4$; see, for example, \cite{Gabai2021,Schwartz2021,KosanovicTeichner2024Space,KosanovicTeichner2024,ChaKim2025,LinWuXieZhang2026}. 

The proof of the Dax isomorphism depends crucially on several family-transversality results in smooth manifolds. Therefore, the argument does not directly extend to the case of topological embeddings.
In this article, we prove a Dax isomorphism for the space of \emph{topologically} embedded arcs in a \emph{topological} $4$-manifold. This allows potential applications of the Dax isomorphism to isotopy problems in topological $4$-manifolds, where the fundamental group is not necessarily ``good'' in the sense of Freedman--Quinn.

Let $X$ be a connected oriented topological $4$-manifold with non-empty boundary. The manifold $X$ is not required to be compact.  Recall that every topological $3$-manifold has a canonical smooth structure up to diffeomorphisms. In the following, we will always fix a specific smooth structure on a neighborhood of $\partial X$, fix two points $p_0,p_1$ on $\partial X$, and all arcs considered are embeddings $f:[0,1]\to X$ such that $f(0)=p_0$, $f(1)=p_1$, and that $f$ has fixed smooth germs at the boundary points which are smoothly transverse to $\partial X$.

Let $I=[0,1]$, let $\Delta^n$ denote the $n$-dimensional simplex.

Let $\Emb^{lf}(I,X)$ denote the simplicial set whose simplices are family-locally-flat embeddings
\[
\Delta^n\times I \to \Delta^n \times X
\]
that preserve the $\Delta^n$-coordinate and satisfy the above boundary conditions.

Let $\Emb^{top}(I,X)$ denote the set of topological embeddings of $I$ in $X$, with the same boundary conditions, endowed with the compact-open topology. Since $I$ is compact, if $X$ is endowed with a compatible metric, then the compact-open topology is the same as the $C^0$-topology.

If $X$ admits a smooth structure and if the smooth structure near $\partial X$ is taken compatibly, let $\Emb^{sm}(I,X)$ denote the space of smooth embeddings of $I$ in $X$, with the same boundary conditions, endowed with the $C^\infty$ topology. Alternatively, we may define $\Emb^{sm}(I,X)$ as the simplicial set whose simplices are given by smooth embeddings 
\[
\Delta^n\times I\to \Delta^n\times X
\]
that preserve the $\Delta^n$ coordinate and agree with the aforementioned boundary conditions. The two spaces are weakly homotopy equivalent, and as far as this paper is concerned, the two definitions can be used interchangeably.

Let  $|\Emb^{lf}(I,X)|$ denote the geometric realization of $\Emb^{lf}(I,X)$.
There is a canonical map from $|\Emb^{lf}(I,X)|$ to $\Emb^{top}(I,X)$. The main result of this paper is the following:

\begin{Theorem}
\label{thm_main}
\begin{enumerate}
\item The map $|\Emb^{lf}(I,X)|\to Emb^{top}(I,X)$ induces isomorphisms on both $\pi_0$ and $\pi_1$. 
\item There exists an exact sequence
\begin{equation}
\label{eqn_top_Dax}
\pi_3(X)\to \mathbb{Z}[\pi_1(X)\setminus\{1\}]\to \pi_1 \Emb^{top}(I,X)\to \pi_2(X)\to 0.
\end{equation}
\item 
The exact sequence \eqref{eqn_top_Dax} is natural with respect to inclusions of $4$-manifolds. More precisely, if $X$ is a codimension-zero submanifold of $X'$ such that the boundary points $p_0$ and $p_1$ are in the interior of $\partial X\cap \partial X'$, then we have the following commutative diagram, where the vertical maps are induced by the inclusion:
\[
\begin{tikzcd}
  \pi_3(X) \arrow[r] \arrow[d] & 
  \mathbb{Z}[\pi_1(X)\setminus\{1\}] \arrow[r] \arrow[d] & 
  \pi_1 \Emb^{top}(I,X) \arrow[r] \arrow[d] & 
  \pi_2(X) \arrow[r] \arrow[d] & 
  0 \arrow[d] \\
  \pi_3(X') \arrow[r] & 
  \mathbb{Z}[\pi_1(X')\setminus\{1\}] \arrow[r] & 
  \pi_1 \Emb^{top}(I,X')  \arrow[r] & 
  \pi_2(X')  \arrow[r] & 
  0.
\end{tikzcd}
\]

\item When $X$ is smooth, the canonical embedding from $\Emb^{sm}(I,X)$ to $\Emb^{top}(I,X)$ induces isomorphisms on $\pi_0$ and $\pi_1$, and \eqref{eqn_top_Dax} agrees with the classical Dax isomorphism for smooth $4$-manifolds.
\end{enumerate}
\end{Theorem}

As an immediate consequence, the classical Dax isomorphism does not depend on the smooth structure of $X$. This answers a question asked by Jae Choon Cha. Theorem \ref{thm_main} also provides an alternative approach for the definition of the topological Dax invariant by Cha--Kim \cite[Theorem F]{ChaKim2025}.

\begin{remark}
The base points in \eqref{eqn_top_Dax} are defined in the same way as the classical Dax isomorphism for smooth manifolds: Choose a $f_0\in \Emb^{top}(I,X)$ as the base point for $\pi_1\Emb^{top}(I,X)$. The base point for the homotopy groups of $X$ is then taken to be an arbitrary point on the image of $f_0$. 
\end{remark}

\begin{remark}
    If $X$ is smooth, then the fact that the maps 
    \begin{equation}
    \label{eqn_sm_to_lf}
    \Emb^{sm}(I,X)\to \Emb^{lf}(I,X)
    \end{equation}
    and 
    \begin{equation}
    \label{eqn_sm_to_top}
    \Emb^{sm}(I,X) \to Emb^{top}(I,X)
    \end{equation}
    both induce isomorphisms on $\pi_0$ and $\pi_1$ are, to some extent, already known in the literature.\footnote{Here, we view the space $\Emb^{sm}(I,X)$ as a simplicial set in \eqref{eqn_sm_to_lf}  and as the set of smooth embeddings in $C^\infty$-topology in \eqref{eqn_sm_to_top}. Recall that the two spaces are weakly homotopy equivalent, and we therefore use them interchangeably in this paper.} The statement that the map \eqref{eqn_sm_to_top} induces isomorphisms on $\pi_0$ and $\pi_1$ essentially follows from the work of Robinson  \cite{robinson1971homotopy, Robinson}. Although Robinson only stated the results for embedding spaces of \emph{closed} manifolds, the arguments there also apply to manifolds with boundary. The fact that \eqref{eqn_sm_to_lf} induces isomorphisms on $\pi_0$ and $\pi_1$ will follow as a consequence of an upcoming work by Kremer--Kupers \cite{Kupers}. 
    
    Nevertheless, in Section \ref{sec_smooth_isomorphism}, we will present self-contained proofs of these results. There are two reasons for doing this: First, although the main ideas already exist in the literature, there do not appear to be explicitly given statements that directly yield the results needed here. Second, the works of Robinson and Kremer--Kupers study embedded spaces in more general dimensions. But even the special case of embeddings of $1$-manifolds into $4$-manifolds is very useful. For example, it allows one to generalize many recent applications of Dax invariants \cite{Gabai2021,Schwartz2021,KosanovicTeichner2024Space,KosanovicTeichner2024,ChaKim2025,LinWuXieZhang2026} to the topological category. Our proof focuses on this case, where the argument can be made significantly simpler and more geometric.  
\end{remark}

\begin{remark}
    The fact that $|\Emb^{lf}(I,X)|\to Emb^{top}(I,X)$ induces isomorphisms on both $\pi_0$ and $\pi_1$ depends heavily on the particular dimensions of the setup of the problem. For example, if we consider the embedding spaces of $S^1$ in $S^3$, then the trefoil and the unknot are in the same connected component of $\Emb^{top}(S^1,S^3)$ but not in the same connected component of $\Emb^{lf}(S^1,S^3)$.
\end{remark}

\begin{remark}
    The maps \eqref{eqn_sm_to_lf} and \eqref{eqn_sm_to_top} do not induce isomorphisms on $\pi_2$. In fact, by embedding calculus, it is known that $\pi_2\Emb^{sm}(I,D^4)\cong \bZ$ (see \cite[Proposition 3.9]{budney2008family}), but we have $\pi_2\Emb^{lf}(I,D^4)=0$ (see Lemma \ref{lem_contractible_Emb_lf_D4} below). Therefore, both \eqref{eqn_sm_to_lf} and \eqref{eqn_sm_to_top} induce non-injective maps on $\pi_2$. An explicit geometric description of the generator of $\pi_2\Emb^{sm}(I,D^4)$ is given in \cite{budney2008family} after the statement of Corollary 3.10.
\end{remark}

\textbf{Acknowledgments:} 
We would like to express our sincere gratitude to Jae Choon Cha for tremendously helpful discussions throughout the development of this paper. We thank Alexander Kupers for earlier correspondence, which allowed us to learn about their upcoming related work.

J. Lin is partially supported by  National Key R\&D Program of China 2025YFA1017500 and NSFC 12271281. 
Y. Xie is partially supported by NSFC 12341105.
B. Zhang is partially supported by NSF grants DMS-2540516, DMS-2405271, and a travel grant from the Simons Foundation. \\

\textbf{AI Statement:} During the preparation of this work, we used AI tools to translate \cite{cerf1961topologie} from French to English for easier reading. Some of the techniques we learned there were used in the proof of Proposition \ref{prop_sm_to_top_iso}, where attribution is included. Apart from this, no AI tools were used for this work.

\section{Isomorphisms on $\pi_0\Emb(I,X)$ and $\pi_1\Emb(I,X)$ when $X$ is smooth}
\label{sec_smooth_isomorphism}

In this section, we assume that $X$ is smooth. We also assume that $X$ is endowed with a smooth Riemannian metric, which induces a metric space structure on $X$. 

All isotopies and homotopies are assumed to fix a neighborhood of the boundary. 
For $\epsilon>0$, we say that an isotopy (resp. homotopy) is an \emph{$\epsilon$-isotopy} (resp. \emph{$\epsilon$-homotopy}), if the trajectory of each point has diameter less than $\epsilon$. In this case, we also say that the \emph{diameter} of the isotopy (resp. homotopy) is less than $\epsilon$.

We will prove the following result.
\begin{Proposition}
\label{prop_sm_to_lf_surj}
  The map  $\Emb^{sm}(I,X)\to \Emb^{lf}(I,X)$ induces a surjection on $\pi_0$, and induces a surjection on $\pi_1$ for an arbitrary choice of base point in $\Emb^{sm}(I,X)$.
\end{Proposition}

As preparation for Proposition \ref{prop_sm_to_lf_surj}, we establish the following lemmas. 
Let $\Map^{sm}(I,X)$ be the set of smooth maps from $I$ to $X$ with the same boundary condition. 
\begin{Lemma}\label{lem_sm_jet-transversality} Then the inclusion map $\Emb^{sm}(I,X)\hookrightarrow \Map^{sm}(I,X)$ has a dense image. Furthermore, every path 
\[
\gamma:I\to \Map^{sm}(I,X)
\]
with $\gamma(0),\gamma(1)\in  \Emb^{sm}(I,X)$ can be $C^0$-perturbed without moving the end points to a path in $\Emb^{sm}(I,X)$.  
\end{Lemma}
\begin{proof}
Since the dimension of $X$ is greater than twice the dimension of $I$, this follows from standard jet-transversality arguments.    
\end{proof}

\begin{Lemma}
\label{lem_sm_emb_dense}
    The set $\Emb^{sm}(I,X)$ is dense in $\Emb^{top}(I,X)$ in the $C^0$-topology.
\end{Lemma}
\begin{proof}
By standard smoothing arguments, the space $\Map^{sm}(I,X)$ is dense in $\Emb^{top}(I,X)$. By Lemma \ref{lem_sm_jet-transversality}, the space $\Emb^{sm}(I,X)$ is dense in $\Map^{sm}(I,X)$.
\end{proof}

\begin{Lemma}
\label{lem_perturb_lf_to_sm}
    For every $f\in \Emb^{lf}(I,X)$ and $\epsilon>0$, there exists $\delta>0$, such that for every $g$ in the set $\{g\in \Emb^{sm}(I,X)\mid \dist_{C^0}(f,g)\}<\delta$, the arc $g$ is ambiently topologically $\epsilon$-isotopic to $f$. 
    
    Moreover, the value of $\delta$ can be taken to depend continuously on $\epsilon$ and $f$ (with respect to the $C^0$ topology).
\end{Lemma}

\begin{proof}
    By Quinn \cite[Theorem 2.2.2]{quinn1982ends}, for every $\delta$, the arc $f$ can be ambiently $\delta$-isotoped to a smooth arc $f'$. 
    Choose $\delta$ sufficiently small such that for every $x\in \Ima(f)$ and $x_{1},x_{2}\in X$ with $\dist(x,x_{i})<\delta$, there is a unique minimal geodesic connecting $x_1$ and $x_2$. The upper bound on such $\delta$ can be taken to depend continuously on $f$. Therefore, if $g\in \Emb^{sm}(I,X)$ satisfies $\dist_{C^0}(f,g)<\delta$, then the arc $g$ can be homotoped to $f'$ by the minimal geodesic homotopy, and the diameter of the homotopy is less than $2\delta$. By Lemma \ref{lem_sm_jet-transversality}, the minimal geodesic homotopy can be perturbed to a smooth isotopy of embedded smooth arcs, which then extends to a $3\delta$-ambient isotopy from $g$ to $f'$ by the isotopy extension theorem. As a result, we obtain a $4\delta$-ambient isotopy from $f$ to $g$. If we further require that $\delta<\epsilon/4$, then the desired result is proved. 
\end{proof}

\begin{Lemma}
\label{lem_contractible_Emb_lf_D4}
    $\Emb^{lf}(I,D^4)$ is contractible.
\end{Lemma}
\begin{proof}
    By Lemmas \ref{lem_sm_emb_dense} and \ref{lem_perturb_lf_to_sm}, we know that every vertex of $\Emb^{lf}(I,D^4)$ is connected to a vertex represented by a smooth arc via a $1$-simplex. Every pair of smoothly embedded arcs is related to each other by a smooth homotopy, which can then be perturbed to a smooth isotopy by transversality. As a consequence,  $\Emb^{lf}(I,D^4)$ is connected.

    Let $\Homeo^+(D^4)$ denote the simplicial group of orientation-preserving homeomorphisms of $D^4$ (relative to the boundary), where the simplices are orientation-preserving homeomorphisms
    \[
    \Delta^n \times D^4 \to \Delta^n \times D^4
    \]
    that preserve the $\Delta^n$ coordinate and equal to the identity on a neighborhood of $\Delta^n\times \partial D^4$.
    Since $\Emb^{lf}(I,D^4)$ is connected, by the topological isotopy extension theorem, the simplicial group $\Homeo^+(D^4)$ acts transitively on  $\Emb^{lf}(I,D^4)$. Let $I_0$ be the standard arc in $D^4$, and let $\Homeo^+(D^4,I_0)$ denote the subgroup of $\Homeo^+(D^4)$ that fixes $I_0$ pointwise. Then 
    \[
    \Emb^{lf}(I,D^4)\cong \Homeo^+(D^4)/\Homeo^+(D^4,I_0).
    \]
    Both  $\Homeo^+(D^4)$ and $\Homeo^+(D^4,I_0)$ are contractible by the Alexander trick. Therefore, $\Emb^{lf}(I,D^4)$ is contractible.
\end{proof}

Now we prove Proposition \ref{prop_sm_to_lf_surj}.

\begin{proof}[Proof of Proposition \ref{prop_sm_to_lf_surj}]
Lemmas \ref{lem_sm_emb_dense} and \ref{lem_perturb_lf_to_sm} yield the surjectivity on $\pi_0$.
We only need to show the surjectivity on $\pi_1$. In the following, we will identify $\Delta^1$ with $[0,1]$, and represent elements in $\pi_1$ by $1$-simplices. 

Suppose a smooth arc $f_0:I\to X$ is the base point. 
Let $\sigma$ be a 1-simplex in $\Emb^{lf}(I,X)$ that represents an element of $\pi_1\Emb^{lf}(I,X)$. That is, $\sigma:\Delta^1 \times I \to \Delta^1\times X$ is a family-locally-flat embedding such that $f|_{\{0\}\times I}$ and $f|_{\{1\}\times I}$ are both given by $f_0$. After homotopy in $\Emb^{lf}(I,X)$, we may assume that $f|_{\{t\}\times I}$ is equal to the base point $f_0$ when $t$ is close to $0$, $1$. We show that $[\sigma]$ is homotopic to a $1$-simplex represented by a smooth map.

By smoothing, for each $\delta > 0$, we may find a smooth map
\[
\sigma' :\Delta^1\times [0,1]\to \Delta^1 \times X
\]
that preserves the $\Delta^1$ coordinate, is equal to $\sigma$ near the boundary $\partial (\Delta^1\times [0,1])$, such that $\dist_{C^0}(\sigma, \sigma')<\delta$. 
By transversality, we may perturb $\sigma'$ fiberwise and assume that it is a smooth embedding.  

Choose $\epsilon>0$ so that for every $t\in \Delta^1$, there exists a closed neighborhood $U_t$ of $\Ima \sigma|_{\{t\}\times [0,1]}$ homeomorphic to $D^4$ such that $U_t$ includes the $\epsilon$-neighborhood of $\Ima \sigma|_{\{t\}\times [0,1]}$. Such an $\epsilon$ exists because of the topological tubular neighborhood theorem for locally flat embeddings in $4$-manifolds and the compactness of $\Delta^1$. 

Decompose the interval $\Delta^1$ into sub-intervals bounded by $0=t_0<t_1<\dots<t_m=1$, where $t_i = i/m$. 
By Lemma \ref{lem_perturb_lf_to_sm}, we may choose $m$ sufficiently large and $\delta$ sufficiently small such that the following statements hold:
\begin{enumerate}
\item For each $i$, there exist ambient $ \epsilon/4$-isotopies from $\sigma|_{\{t_i\}\times I}$ to $\sigma'|_{\{t_i\}\times I}$. Moreover, the isotopy is constant with respect to time when $i=0$ or $m$.
\item For each $i$, the 1-parameter family of arcs $\sigma|_{\{t\}\times I}$ for $t\in[t_i,t_{i+1}]$ extends to an $\epsilon/4$-ambient isotopy. Similarly, the 1-parameter family of arcs $\sigma'|_{\{t\}\times I}$ for $t\in[t_i,t_{i+1}]$ extends to an $\epsilon/4$-ambient isotopy.
\end{enumerate}

Let $\xi_i$ denote the $\epsilon/4$-isotopy from $\sigma|_{\{t_i\}\times I}$ to $\sigma|_{\{t_{i+1}\}\times I}$.
Let $\xi_i'$ denote the $\epsilon/4$-isotopy from $\sigma'|_{\{t_i\}\times I}$ to $\sigma'|_{\{t_{i+1}\}\times I}$.
Let $\tau_i$ denote the $ \epsilon/4$-isotopy from $\sigma|_{\{t_i\}\times I}$ to $\sigma'|_{\{t_i\}\times I}$. Then for each $i$, the loop formed by $\xi_i$, $\tau_{i+1}$, $\overline{\xi_i'}$, $\overline{\tau_i'}$ is a family-locally-flat loop in the space of locally flat embeddings of arcs in $X$, whose images are included in $U_{t_i}$. By Lemma \ref{lem_contractible_Emb_lf_D4}, this loop is contractible in $\Emb^{lf}(I,U_{t_i})$, and hence is contractible in $\Emb^{lf}(I,X)$. As a result, the loop extends to a map from $D^2$ to $\Emb^{lf}(I,X)$. Therefore, there exists a family-locally-flat embedding 
\[
\xi: [0,1]\times \Delta^1\times [0,1]\to [0,1]\times \Delta^1\times X
\]
that preserve the $([0,1]\times \Delta^1)$ coordinate, such that $\xi|_{\{0\}\times \Delta^1\times [0,1]}=\sigma$, $\xi|_{\{1\}\times \Delta^1\times [0,1]}=\sigma'$. As a result, the 1-simplices in $\Emb^{lf}(I,X)$ represented by $\sigma'$ and $\sigma$ are homotopic relative to the base point. 
\end{proof}

Now we prove the following proposition.
\begin{Proposition}[{See also \cite{Robinson}}]
\label{prop_sm_to_top_iso}
    $\Emb^{sm}(I,X)\to \Emb^{top}(I,X)$ induces isomorphisms on $\pi_0$ and $\pi_1$. 
\end{Proposition}
The proof is essentially a simplification of the arguments of Cerf \cite{cerf1961topologie} and Robinson \cite{robinson1971homotopy,Robinson} for the special case of embedded arcs in $4$-manifolds. We start with the following two lemmas.
\begin{Lemma}
\label{lem_Emb_D4_trivial}
$\pi_1\Emb^{sm}(I,D^4)$ is trivial. 
\end{Lemma}

\begin{proof}
    This follows immediately from the classical Dax isomorphism.
\end{proof}

\begin{Lemma}[{see also \cite[Chapter 1]{robinson1971homotopy}}]
\label{lem_local_extension_sm}
    For each $f\in \Emb^{top}(I,X)$ and each $\epsilon > 0 $, there exists $\delta>0$, such that for $i=0,1$, every map $S^i\to \Emb^{sm}(I,X)$ within the $\delta$-neighborhood of $f$ (with respect to the $C^0$ topology) can be extended to a map $D^{i+1}\to \Emb^{sm}(I,X)$ within the $\epsilon$-neighborhood of $f$. 

    Moreover, $\delta$ can be taken to depend continuously on $\epsilon$ and $f$.
\end{Lemma}
\begin{proof}
Choose $\delta$ sufficiently small so that every pair of smooth arcs in the $\delta$-neighborhood of $f$ admits a unique-minimal-geodesic homotopy in the $\epsilon/2$-neighborhood of $f$. We may perturb the homotopy so that it becomes a smooth isotopy in the $\epsilon/2$-neighborhood of $f$. This finishes the proof for $i=0$.

For $i=1$, let $\Imm^{sm}(I,X)$ denote the space of smooth immersions with the same boundary conditions as the embedding spaces, and suppose we are given a map $\sigma: S^1\to \Emb^{sm}(I,X)$ in the $\delta$-neighborhood of $f$. 
By perturbing the unique-minimal-geodesic homotopy, we obtain a map 
\[
\xi: D^2\to \Imm^{sm}(I,X),
\]
such that 
\begin{enumerate}
    \item $\xi|_{\partial D^2} = \sigma$,
    \item There exist finitely many points $p_1,\dots,p_m$ in the interior of $D^2$, where each $\xi(p_i)$ is an immersed arc with a clean double point $x_i\in X$,
    \item For every $p\notin \{p_1,\dots,p_m\}$, the arc $\xi(p)$ is embedded. 
\end{enumerate}

Let 
\[
\delta' = \sup_{\dist(f(t_1),f(t_2))\le 2\delta} |t_1-t_2|,
\]
and let 
\[
\delta'' = \sup_{|t_1-t_2|\le \delta'} \dist(f(t_1),f(t_2)).
\]
Since $f$ is an injection and is continuous, the values of $\delta'$ and $\delta''$ converge to zero as $\delta\to 0$. 

For each $p_i$, the arc $\xi(p_i)$ has a unique self-intersection point, so its image contains a unique loop. The loop has diameter at most $\delta''$. We may choose $\delta$ sufficiently small so that every embedded loop with diameter at most $\delta''$ in the $\epsilon/2$-neighborhood of $\Ima(f)$ bounds an embedded $2$-disk with diameter at most $2\delta''$. Consider a $2$-disk $D^{(i)}$ bounded by the loop in the image of $\xi(p_i)$. After perturbation, we may require that the interior of $D^{(i)}$ is disjoint from $\xi(p_i)$. Let $\nu(D^{(i)})$ be a closed neighborhood of $D^{(i)}$ diffeomorphic to $D^4$. We may require that the diameter of  $\nu(D^{(i)})$ is less than $3\delta''$, and that the intersection of $\nu(D^{(i)})$ with $\xi(p_i)$ has one connected component.
Perturbing the boundary of $\nu(D^{(i)})$, we may assume that it intersects $\xi(p_i)$ transversely. By Lemma \ref{lem_Emb_D4_trivial}, we conclude that the map $\xi$ can be modified in an arbitrarily small neighborhood $U_i$ of $p_i$ so that $\xi(U_i)$ only contains embedded arcs, and that the $C^0$-distance between $\xi$ and the modified $\xi$ is bounded by $3\delta''$. Choose the $U_i$'s so that they are mutually disjoint, and choose $\delta$ such that $3\delta''<\epsilon/2$. Modifying $\xi$ on all $U_i$'s as above then yields the desired extension map.  
\end{proof}

\begin{proof}[Proof of Proposition \ref{prop_sm_to_top_iso}]

{\bf Surjectivity on $\pi_0$:} For every topologically embedded arc $f\in \Emb^{top}(I,X)$, by smoothing and transversality, there is a sequence $f_n\in \Emb^{sm}(I,X)$ that converges to $f$ in $C^0$. After taking a subsequence, we may assume that $\dist_{C^0}(f_n,f)<1/2^n$. There exists $N$ such that for all $i\ge N$, there is a unique-minimal-geodesic homotopy from $f_i$ to $f_{i+1}$; by transversality, after perturbation, this gives a path in $\Emb^{sm}(I,X)$ from $f_i$ to $f_{i+1}$ with $C^0$-length at most $2\dist_{C^0}(f_i,f_{i+1})$. The concatenation of these paths for all $i\ge N$ gives a continuous path from $f_{n_0}$ to $f$ in $\Emb^{top}(I,X)$. 

\bigskip 
{\bf Injectivity on $\pi_0$:} Suppose $f_0, f_1\in \Emb^{sm}(I,X)$, and $f_t\in \Emb^{top}(I,X)$ $(t\in [0,1])$ is a path connecting $f_0$ and $f_1$. After homotopy in $\Emb^{top}(I,X)$, we may assume that $f_t$ is locally constant with respect to $t$ near $t=0$ and $1$. By smoothing, there is a smooth family $\{f_t'\}$ of (not necessarily injective) maps $I\to X$ with the same boundary condition, such that $f_t' = f_t$ when $t$ is close to $0$ and $1$. By transversality, the smooth family $\{f_t'\}$ can be perturbed to a smooth family of embedded arcs with the same boundary condition and the same values near $t=0$ and $1$. Therefore, $f_0$ and $f_1$ are connected in $\Emb^{sm}(I,X)$. 

\bigskip 
{\bf Surjectivity on $\pi_1$:} Let $f_0\in \Emb^{sm}(I,X)$ be the base point. Suppose $f_t\in \Emb^{top}(I,X)$ $(t\in [0,1])$ is a path such that $f_1=f_0$. We show that the element in $\pi_1 \Emb^{top}(I,X)$ represented by $\{f_t\}$ is in the image of $\pi_1 \Emb^{sm}(I,X)$. After homotopy in $\Emb^{top}(I,X)$, we may assume that $f_t=f_0=f_1$ near $t=0$ and $1$. By smoothing and transversality, we obtain a sequence of smooth families of embedded arcs $\{f_t^{(n)}\}$ ($n\ge 1$) with the same boundary conditions and the same values near $t=0$ and $t=1$ such that the sequence $\{f_t^{(n)}\}_{n\ge 1}$ converges to $f_t$ uniformly with respect to $t$ in $C^0$.

By decomposing the interval $t\in [0,1]$ into sufficiently small sub-intervals and applying Lemma \ref{lem_local_extension_sm}, we obtain that for each $\epsilon > 0$, there exists $n(\epsilon)$, such that for each $n\ge n(\epsilon)$, the path $\{f_t^{(n)}\}_{t\in[0,1]}$ is homotopic to $\{f_t^{(n+1)}\}_{t\in[0,1]}$ in $\Emb^{sm}(I,X)$ relative to $t\in \{0,1\}$ via a homotopy within the $\epsilon$-neighborhood of the path $\{f_t\}_{t\in[0,1]}$. 

Taking a sequence of values of $\epsilon$'s converging to zero, we obtain an integer $N$ and a sequence of homotopies from $\{f_t^{(n)}\}_{t\in[0,1]}$ to $\{f_t^{(n+1)}\}_{t\in[0,1]}$ in $\Emb^{sm}(I,X)$ for all $n\ge N$, such that the $C^0$-distance between the homotopy from $\{f_t^{(n)}\}_{t\in[0,1]}$ to $\{f_t^{(n+1)}\}_{t\in[0,1]}$ and the constant homotopy at $\{f_t\}_{t\in[0,1]}$ converges to zero as $n\to \infty$. The concatenation of these homotopies yields a homotopy from $\{f_t^{(N)}\}$ to $\{f_t\}$ in $\Emb^{top}(I,X)$ relative to $t\in \{0,1\}$.

\bigskip 
{\bf Injectivity on $\pi_1$:} Suppose there is a map $\sigma: S^1\to \Emb^{sm}(I,X)$ that extends to $\xi:D^2\to \Emb^{top}(I,X)$. We need to construct another map $\xi':D^2\to \Emb^{sm}(I,X)$ with the same boundary as $\xi$. First, after homotopy in $\Emb^{top}(I,X)$, we may assume that the image of $\xi$ on a neighborhood of $\partial D^2$ is contained in $\Emb^{sm}(I,X)$. Now, identify $D^2$ with $I\times I$, and divide $I\times I$ into $m\times m$ sub-squares. Let $L$ denote the $1$-skeleton of the subdivision. By smoothing and transversality, for each $\delta>0$, we may perturb $\xi|_L$ to a map $\xi'$ on $L$, such that the image of $\xi'$ is contained in $\Emb^{sm}(I,X)$, and that the $C^0$ distance between $\xi'$ and $\xi|_L$ is less than $\delta$, and that $\xi=\xi'$ on $\partial (I\times I)$.   By Lemma \ref{lem_local_extension_sm}, for $m$ sufficiently large and $\delta$ sufficiently small, the map $\xi'$ extends over the $2$-cells to give a map $I\times I\to \Emb^{sm}(I,X)$ with the same boundary as $\xi$. 
\end{proof}

As an immediate corollary of Propositions \ref{prop_sm_to_lf_surj} and \ref{prop_sm_to_top_iso}, we have:

\begin{Corollary}
    The canonical maps $\Emb^{sm}(I,X)\to \Emb^{lf}(I,X)$ and $|\Emb^{lf}(I,X)|\to \Emb^{top}(I,X)$ induce isomorphisms on $\pi_0$ and $\pi_1$. 
\end{Corollary}

\section{Independence of the smooth structure}

Suppose $X$ admits a smooth structure. 
The Dax isomorphism states that there is a long exact sequence 
\begin{equation}
\label{eqn_geometric_Dax}
\pi_3(X)\to \bZ[\pi_1(X)\setminus\{1\}]\to \pi_1\Emb^{sm}(I,X)\to \pi_2(X)\to 0.
\end{equation}
An explicit geometric description of the maps in \eqref{eqn_geometric_Dax} is given by Gabai in \cite[Theorem 3.7]{Gabai2021}. The maps in \eqref{eqn_geometric_Dax}, as well as the proof of its exactness, depend on family-transversality results for maps between smooth manifolds, and hence a priori depend on the smooth structure.

Now suppose that $X$ admits two smooth structures $\tau_1$ and $\tau_2$. Suppose $f_0\in \Emb^{top}(I,X)$ is smooth with respect to both $\tau_1$ and $\tau_2$, and let $f_0$ be the base point of the embedding spaces. Proposition \ref{prop_sm_to_top_iso} shows that we have canonical isomorphisms
\[
\pi_1\Emb^{sm}(I,X_{\tau_1})\to \pi_1\Emb^{top}(I,X) \leftarrow \pi_1\Emb^{sm}(I,X_{\tau_2}).
\]
Therefore, by identifying $\pi_1\Emb^{sm}(I,X_{\tau_1})$ with $\pi_1\Emb^{sm}(I,X_{\tau_2})$ using the above isomorphism, it makes sense to compare the Dax isomorphisms with respect to $\tau_1$ and $\tau_2$. The main result of this section is the following:

\begin{Proposition}
\label{prop_Dax_independent_smooth_str}
    Let $\tau_1$, $\tau_2$, and the arc $f_0$ be as above. Then the exact sequences \eqref{eqn_geometric_Dax} defined with respect to $\tau_1$ and $\tau_2$ are the same. 
\end{Proposition}

\begin{remark}
    Note that the independence of the isomorphism classes of 
    \[
    \pi_i\Emb^{sm}(I,X), \quad (i=0,1)
    \]
    on the smooth structure already follows from Proposition \ref{prop_sm_to_top_iso}; the new content of Proposition \ref{prop_Dax_independent_smooth_str} is that all the homomorphisms arising from the Dax isomorphism also do not depend on the smooth structure. 

    If $M$ is a smooth closed $4$-manifold, the result of \cite{arone2022spaces} implies that 
    \[
    \pi_i\Emb^{sm}(S^1,X), \quad (i=0,1)
    \]
    is independent of the smooth structure. If $M$ is a simply connected, compact, smooth $4$-manifold, \cite{knudsen2024embedding} further proved that the homotopy type of  $\Emb^{sm}(\sqcup_m S^1,M)$ is independent of the smooth structure. 
\end{remark}

\subsection{An alternative description of the Dax isomorphism}
\label{subsec_Dax_by_emb_calculus}

We first present an equivalent formulation of the Dax isomorphism using embedding calculus (see also \cite[Section 1.1]{kosanovic2024knotted}). 

Suppose $X$ is a smooth $4$-manifold.
Let $C_n\langle X, \partial \rangle$ denote the simplicial compactification of the $n$-point configuration space of $X$, as defined by \cite[Definition 4.10]{sinha2009topology}\footnote{Our notation $C_n\langle X,\partial \rangle$ is equivalent to the notation $C_n\langle[X,\partial]\rangle$ in \cite{sinha2009topology}.}, with respect to the two fixed boundary points in the definition of the embedding spaces of arcs. We refer the reader to the appendix of \cite{lin2025mapping} for an exposition on the definition and basic properties of the simplicial compactification of the configuration space.

Let $C_n'\langle X, \partial \rangle$ be the $(S^3)^n$ bundle over $C_n\langle X, \partial \rangle$, where the fiber over each point is given by the product of the spheres of unit tangent vectors at all points in the configuration (see \cite[Definitions 4.9, 4.10]{sinha2009topology}). The definitions of both $C_n\langle X, \partial \rangle$ and $C_n'\langle X, \partial \rangle$ depend on the smooth structure of $X$.  The spaces $\{C_n'\langle X,\partial \rangle\}_{n\ge 0}$ form a cosimplicial space. Moreover, the cosimplicial structure is Reedy cofibrant: all coface maps, as well as the inclusion of the union of images of coface maps to $C_n'\langle X,\partial \rangle$ as a subset into $C_n'\langle X,\partial \rangle$, are cofibrations. For each $C_n'\langle X,\partial \rangle$, the images of all coface maps define a stratification on $C_n'\langle X,\partial \rangle$, which is indexed by the faces of $\Delta^n$.

Let 
\[
\Map_\Delta(\Delta^n,C_n'\langle X,\partial \rangle)
\]
denote the space of continuous maps that preserve the stratification structure indexed by the faces of $\Delta^n$, with the compact-open topology\footnote{The space $\Map_\Delta(\Delta^n,C_n'\langle X,\partial \rangle)$ is denoted by $\Map_n(M)$ in \cite{lin2025mapping}.}. Then for each $n$, we have a map
\begin{equation}
\label{eqn_sm_emb_to_configuration_space}
\Emb^{sm}(I,X) \to \Map_\Delta(\Delta^n,C_n'\langle X,\partial \rangle),
\end{equation}
which is defined by taking the induced map on the configuration spaces for each embedding.
By \cite[Theorem 5.4, Lemma 5.12]{sinha2009topology}, we know that \eqref{eqn_sm_emb_to_configuration_space} induces an isomorphism on $\pi_1$ when $n\ge 2$. 

By \cite[Section 5]{lin2025mapping}, a version of the Bousfield--Kan spectral sequence gives the following long exact sequence
\begin{multline}
\label{eqn_long_exact_Map_Delta2}
\oplus_3 \pi_3C_1'\langle X,\partial\rangle \to  \pi_3C_2'\langle X,\partial\rangle \to \pi_1 \Map_{\Delta}(\Delta^2,C_2'\langle X,\partial\rangle) 
\\
\to \oplus_3 \pi_2 C_1'\langle X,\partial\rangle \to \pi_2 C_2'\langle X,\partial \rangle \to 0,
\end{multline}
where 
\begin{align*}
 \pi_3C_1'\langle X,\partial\rangle &\cong \pi_3(X)\oplus \bZ,\\
\pi_3C_2'\langle X,\partial\rangle &\cong  (\pi_3(X)\oplus \bZ)^2 \oplus \mathbb{Z}[\pi_1(X)],\\
\pi_2 C_1'\langle X,\partial \rangle &\cong \pi_2(X),\\
\pi_2C_2'\langle X,\partial\rangle &\cong \pi_2(X)\oplus \pi_2(X) .
\end{align*}
The homomorphism 
\[
\oplus_3 \pi_3C_1'\langle X,\partial\rangle\to \pi_3C_2'\langle X,\partial\rangle
\]
takes two factors of $\pi_3(X)\oplus \bZ$ isomorphically to the two factors of $\pi_3(X)\oplus \bZ$ in $\pi_3C_2'\langle X,\partial\rangle$, and takes the other $\bZ$ factor to the generator of $\bZ[\pi_1(X)]$ given by $1\in \pi_1(X)$.

The homomorphism 
\[
\oplus_3 \pi_2(X)\cong \oplus_3 \pi_2 C_1'\langle X,\partial\rangle \to \pi_2 C_2'\langle X,\partial \rangle\cong \pi_2(X)\oplus \pi_2(X)
\]
takes $(a,b,c)$ to $(a-c,b-c)$.

As a result, we obtain a long exact sequence
\begin{equation}\label{eqn_exact_sequence_C2'}
\pi_3(X) \to \bZ[\pi_1(X)\setminus\{1\}]\to \pi_1\Map_{\Delta}(\Delta^2,C_2'\langle X,\partial\rangle)) \to \pi_2(X)\to 0.
\end{equation}
This, together with the fact that \eqref{eqn_sm_emb_to_configuration_space} induces an isomorphism on $\pi_1$ when $n=2$, gives an exact sequence
\begin{equation}
\label{eqn_sm_Dax_iso}
\pi_3(X) \to \bZ[\pi_1(X)\setminus\{1\}]\to \pi_1\Emb^{sm}(I,X) \to \pi_2(X)\to 0.
\end{equation}

\begin{Proposition}
    The exact sequence \eqref{eqn_sm_Dax_iso} is the same as the Dax isomorphism \eqref{eqn_geometric_Dax} described by Gabai in \cite[Theorem 3.7]{Gabai2021}.
\end{Proposition}

\begin{proof}
We explicitly describe the maps in \eqref{eqn_long_exact_Map_Delta2} and \eqref{eqn_exact_sequence_C2'} and compare with the maps in \cite[Theorem 3.7]{Gabai2021}. The descriptions below follow from straightforward diagram chasing in the Bousfield--Kan spectral sequence. 

To differentiate notation, we will call the exact sequence described in \cite[Theorem 3.7]{Gabai2021} the \emph{geometric} Dax isomorphism.

The homomorphism $\pi_1 \Map_{\Delta}(\Delta^2,C_2'\langle X,\partial\rangle) 
\to \oplus_3 \pi_2 C_1'\langle X,\partial\rangle$ in \eqref{eqn_long_exact_Map_Delta2} is given by taking the restriction maps of $\Map_{\Delta}(\Delta^2,C_2'\langle X,\partial\rangle)$ to the three edges. Therefore, the map $\pi_1\Map_{\Delta}(\Delta^2,C_2'\langle X,\partial\rangle)) \to \pi_2(X)$ in \eqref{eqn_exact_sequence_C2'} is given by taking the homotopy class of the restriction to any one of the edges. As a result, the composition
\[
\pi_1\Emb^{sm}(I,X)\to \pi_1\Map_{\Delta}(\Delta^2,C_2'\langle X,\partial\rangle)) \to \pi_2(X)
\]
is given by taking the homotopy class of the trajectory, which agrees with the corresponding map in the geometric Dax isomorphism.

Before comparing the other maps in the exact sequences, note that there are two forgetful maps from $C_2'\langle X,\partial \rangle$ to $C_1'\langle X,\partial \rangle$, and the kernel of 
\begin{equation}
\label{eqn_pi3C2'_to_two_pi3C1'}
\pi_3C_2'\langle X,\partial \rangle\to \pi_3C_1'\langle X,\partial \rangle\oplus \pi_3C_1'\langle X,\partial \rangle
\end{equation}
is isomorphic to $\bZ[\pi_1(X)]$ (see \cite[Section 4]{lin2025mapping}). We identify $\bZ[\pi_1(X)\setminus\{1\}]$ with $\bZ[\pi_1(X)]/\bZ$.

The map $\bZ[\pi_1(X)\setminus\{1\}]\to \pi_1\Map_{\Delta}(\Delta^2,C_2'\langle X,\partial\rangle))$ is defined by first taking a lift of an element $\lambda\in \bZ[\pi_1(X)\setminus\{1\}]$ to $\pi_3C_2'\langle X,\partial \rangle$, then gluing it into the interior of the trivial element of $\pi_1\Map_{\Delta}(\Delta^2,C_2'\langle X,\partial\rangle))$, which is represented by a map $[0,1]\times \Delta^2 \to C_2'\langle X,\partial\rangle$. It is straightforward to verify that this map agrees with the composition
\[
\bZ[\pi_1(X)\setminus\{1\}]\to \pi_1\Emb^{sm}(I,X) \to \pi_1 \Map_{\Delta}(\Delta^2,C_2'\langle X,\partial\rangle),
\]
where the first map is given by the geometric Dax isomorphism.

There are three coface maps 
\[
\delta^i_2: C_1'\langle X,\partial \rangle \to C_2'\langle X,\partial \rangle.
\]
The composition homomorphism 
\begin{equation}
\label{composition_pi3_to_Zpi1}
 \pi_3(X)\oplus \bZ\cong C_1'\langle X,\partial \rangle
\xrightarrow{(\delta^0_2)_*-(\delta^1_2)_*+(\delta^2_2)_*} C_2'\langle X,\partial \rangle 
\end{equation}
has image in the kernel of \eqref{eqn_pi3C2'_to_two_pi3C1'}, which is isomorphic to $\bZ[\pi_1(X)]$. Moreover, the composition \eqref{composition_pi3_to_Zpi1} takes the generator of the $\bZ$-component in $\pi_3(X)\oplus \bZ$ to the generator represented by $1\in \pi_1(X)$ in $\bZ[\pi_1(X)]$. Therefore, \eqref{composition_pi3_to_Zpi1} induces a map 
\[
\pi_3(X) \to \bZ[\pi_1(X)\setminus\{1\}].
\]
This describes the first arrow in \eqref{eqn_exact_sequence_C2'}. If an element $\alpha\in \pi_3(X)$ is described by a $D^2$-family of immersed arcs in $X$, then there is a homotopy between the representatives of the images of $(\delta^1_2)_*(\alpha)$ and $(\delta^0_2)_*(\alpha)+(\delta^2_2)_*(\alpha)$ in $\pi_3C_1'\langle X,\partial \rangle \oplus \pi_3C_1'\langle X,\partial \rangle$ under \eqref{eqn_pi3C2'_to_two_pi3C1'}, which is given by moving the points along each immersed arc. The homotopy lifts to $C_2'\langle X,\partial \rangle$ if $\alpha$ is represented by a $D^2$-family of embedded curves. In general, each self-intersecting arc in the $D^2$-family gives rise to a contribution to the difference 
\[
(\delta^0_2)_*(\alpha)-(\delta^1_2)_*(\alpha)+(\delta^2_2)_*(\alpha),
\]
and the contribution agrees with the contribution in the definition of the corresponding map in the geometric Dax isomorphism.
\end{proof}

\subsection{Independence of the smooth structure}
Now we prove Proposition \ref{prop_Dax_independent_smooth_str}, which states that the Dax isomorphism \eqref{eqn_geometric_Dax} for smooth $4$-manifolds does not depend on the smooth structure.

\begin{proof}[Proof of Proposition \ref{prop_Dax_independent_smooth_str}]
Let $X$ be a smooth $4$-manifold. We will show that the reformulation of the Dax isomorphism given in Section \ref{subsec_Dax_by_emb_calculus} admits a further alternative description that does not depend on the smooth structure. The construction is based on the ideas from \cite{budney2025automorphism}. Similar constructions also appeared in \cite[Section 7]{lin2025mapping} and \cite{lin2026pseudo}.

Let $\widetilde{X}$ be the universal cover of $X$.
Let 
\begin{multline*}
C_n^\tau(X) =\{(x_1,\dots,x_n)\in \widetilde{X}^n\mid x_i\neq \eta (x_j) \text{ for all } i\neq j\\
\text{and all non-trivial deck transformations }\eta\}.
\end{multline*}
This is also a cosimplicial set, where the coface maps are defined by doubling points in the configuration, and the codegeneracy maps are defined by the forgetful maps. The cosimplicial structure is again Reedy cofibrant, and there is a map
\[
\Emb^{top}(I,X)\to \Map_{\Delta}(\Delta^n,C_2^\tau(X))
\]
defined by lifting the image of a configuration of points on $I$ to $\widetilde{X}$ using the path-lifting property. 

Consider the commutative diagram
\[
\begin{tikzcd}
  \Emb^{sm}(I,X) \ar[r,"\varphi_1"] \ar[d,"\varphi_2"] & \Emb^{top}(I,X) \ar[d,"\varphi_3"] \\
  \Map_{\Delta}(\Delta^2,C_2'\langle X,\partial\rangle) \ar[r,"\varphi_4"] & \Map_{\Delta}(\Delta^2,C_2^\tau(X))
\end{tikzcd}
\]
By Proposition \ref{prop_sm_to_top_iso}, the map $\varphi_1$ induces an isomorphism on $\pi_1$. Recall that by \cite[Theorem 5.4, Lemma 5.12]{sinha2009topology}, we know that the map $\varphi_2$ induces an isomorphism on $\pi_1$.

We claim that the map $\varphi_4$ also induces an isomorphism on $\pi_1$.
This follows from a direct calculation. Applying the Bousfield--Kan spectral sequence as formulated in \cite[Section 5]{lin2025mapping}, we have a long exact sequence
\begin{multline*}
    \oplus_3 \pi_3C_1^\tau(X) \to  \pi_3C_2^\tau(X) \to \pi_1 \Map_{\Delta}(\Delta^2,C_2^\tau(X))
    \\
    \to \oplus_3 \pi_2 C_1^\tau(X) \to \pi_2 C_2^\tau(X)\to 0.
\end{multline*}
And we have
\begin{align*}
 \pi_3C_1^\tau(X)&\cong \pi_3(X),\\
\pi_3C_2^\tau (X) &\cong  (\pi_3(X))^2 \oplus \mathbb{Z}[\pi_1(X)\setminus\{1\}],\\
\pi_2C_1^\tau(X)&\cong \pi_2(X),\\
\pi_2C_2^\tau(X)&\cong \pi_2(X)\oplus \pi_2(X) .
\end{align*}
The homomorphism 
\[
\oplus_3 \pi_3C_1^\tau(X)\to \pi_3C_2^\tau(X)
\]
takes two factors of $\pi_3(X)$ isomorphically to the two factors of $\pi_3(X)$ in the codomain. 
The homomorphism 
\[
\oplus_3 \pi_2(X) \cong \oplus_3 \pi_2C_1^\tau(X)\to \pi_2C_2^\tau(X)\cong \pi_2(X)\oplus \pi_2(X)
\]
takes $(a,b,c)$ to $(a-c,b-c)$.

As a result, we obtain a long exact sequence
\begin{equation}
\label{eqn_exact_sequence_tau}
\pi_3(X) \to \bZ[\pi_1(X)\setminus\{1\}]\to \pi_1\Map_{\Delta}(\Delta^2,C_2^\tau(X)) \to \pi_2(X)\to 0.
\end{equation}
Moreover, the map $\varphi_4$ induces a homomorphism of chain complexes from \eqref{eqn_exact_sequence_C2'} to \eqref{eqn_exact_sequence_tau}, and the induced maps are identities on all terms other than the term involving the mapping spaces. Hence, by the five lemma, we conclude that $\varphi_4$ induces an isomorphism on $\pi_1$.

Therefore, the Dax isomorphism \eqref{eqn_exact_sequence_C2'} is identified with the long exact sequence \eqref{eqn_exact_sequence_tau} via the map $\varphi_4$. It is clear that \eqref{eqn_exact_sequence_tau}  does not depend on the smooth structure of $X$, because all the spaces and maps are directly defined in the topological category.
\end{proof}

\section{Proof of the main theorem}

Now we can prove Theorem \ref{thm_main}.

\begin{proof}[Proof of Theorem \ref{thm_main}]
Recall that $p_0,p_1\in\partial X$ are fixed points that define the boundary conditions of the arcs, and that we fixed a smooth structure on $\partial X$. Let $f_0:I\to X$ be the base point of $\Emb^{lf}(I,X)$. 

Let $q$ be a different point on $\partial X$, and let $X'=X\setminus\{q\}$.
The inclusion of $X'$ in $X$ induces a map
\[
\varphi_1: \Emb^{top}(I,X')\to  \Emb^{top}(I,X).
\]
On the other hand, the identity embedding of $X$ to $X$ can be homotoped, via topological embeddings fixing a neighborhood of $\Ima(f_0)$, to an embedding of $X$ into $X'$. The resulting embedding yields a map
\[
\varphi_2: \Emb^{top}(I,X)\to  \Emb^{top}(I,X').
\]
It is clear that $\varphi_1\circ\varphi_2$ and $\varphi_2\circ\varphi_1$ are both homotopic to the identity relative to $f_0$. Therefore, the map $\varphi_1$ induces an isomorphism on $\pi_1$. Similarly, the inclusion of $X'$ into $X$ induces an isomorphism 
\[
\pi_1\Emb^{lf}(I,X')\to \pi_1\Emb^{lf}(I,X).
\]

 By \cite[Section 8.2]{freedman2014topology}, the manifold $X'$ admits a smooth structure such that $f_0$ is smooth. By the earlier results, we have a commutative diagram
\[
\begin{tikzcd}
\pi_1^{sm}\Emb(I,X') \arrow[r,"\cong"]  & \pi_1^{lf}\Emb(I,X') \arrow[r,"\cong"]\arrow[d,"\cong"]   & \pi_1^{top}\Emb(I,X') \arrow[d,"\cong"]\\
& \pi_1^{lf}\Emb(I,X) \arrow[r,"\cong"]           
& \pi_1^{top}\Emb(I,X)  
\end{tikzcd}
\]
This proves Part (1) of the theorem.

The (smooth) Dax isomorphism on $X'$ provides a Dax isomorphism for $\pi_1^{top}\Emb(I,X')$, which then yields a Dax isomorphism for $\pi_1^{top}\Emb(I,X)$. By Proposition \ref{prop_Dax_independent_smooth_str}, the Dax isomorphism for $X$ does not depend on the choice of the smooth structure of $X'$. Therefore, if $X$ is itself smooth, then the Dax isomorphism agrees with the smooth Dax isomorphism via the isomorphism in Proposition \ref{prop_sm_to_top_iso}. This proves Parts (2), (4) of the theorem. 

In order to achieve naturality, we need to make a canonical choice of the point $q$. Take $q$ to be a point on the same connected component of $\partial X$ as $p_0$. Then the point $q$ is unique up to ambient isotopy relative to a neighborhood of $\Ima(f_0)$. Therefore, the resulting Dax isomorphism is natural with respect to inclusions of $4$-manifolds. This finishes Part (3) of the theorem.
\end{proof}

\bibliographystyle{amsalpha}
\bibliography{references}

\end{document}